\documentclass[12pt]{article}
\usepackage{amsmath,amssymb,amsthm}
\usepackage[latin1]{inputenc}
\usepackage{xcolor}
\usepackage{todonotes}
\usepackage{epsfig}
\usepackage{color}
\usepackage[T1]{fontenc}

\newtheorem{theorem}{Theorem}

\newtheorem{definition}[theorem]{Definition}
\newtheorem{example}[theorem]{Example}

\newtheorem{proposition}[theorem]{Proposition}
\newtheorem{remark}[theorem]{Remark}

\numberwithin{equation}{section} 
\numberwithin{theorem}{section}

\title{Stochastic dominance of sums of risks under dependence conditions}   

\vspace{2cm}

\author{Jorge Navarro$^{1,}$\footnote{Correspondence to: Jorge Navarro, Facultad
		de Matemáticas, Universidad de Murcia, 30100 Murcia, Spain. E-mail: \textit{jorgenav@um.es}. Telephone number: 34 868883509.
		Fax number: 34 868884182.}
	\quad
	and Jose Miguel Zapata$^1$
	\\
	\quad
	\\
	$^1$Facultad de Matemáticas, Universidad de Murcia,
		Murcia, Spain.
}

\begin{document}
	\maketitle
	
	\begin{abstract}
		
We provide conditions for the stochastic dominance comparisons of a risk $X$ and an associated
risk $X+Z$, where $Z$ represents the uncertainty due to the environment and where $X$ and $Z$ can be dependent. The comparisons depend on both the copula $C$ between the distributions of $X$ and $Z$ and  on the distribution of $Z$. We provide two different conditions for $C$ which represents new positive dependence properties. Regarding $Z$, we need some symmetry or asymmetry (skew) properties. Some illustrative examples are provided.  
		
		\quad
		
		
		\quad
		
		
	\end{abstract}

\section{Introduction}%


Let us consider a risk (or a payoff in an insurance coverage) represented by a random variable $X$ and an associated risk $Y=X+Z$, where we have added a random variable $Z$ which represents some ``noise'' or uncertainty in the practical development of $X$. Here $X$ and $Z$ can be dependent since they share the same environment (usually they are positively dependent). In this situation it is interesting to establish  conditions under which $Y$ is riskier than $X$ under some criteria (since in $Y$ we have added some uncertainty). In this sense, if $u:\mathbb{R}\to \mathbb{R}$ is an \textit{utility function}, we want to compare $E(u(X))$ and  $E(u(Y))$ under some conditions for the function $u$ (increasing, convex, increasing and convex, increasing and concave, etc.). These conditions lead to comparisons in size, in variance or in weights of tails.

Several results have been obtained in the literature to stochastically compare the risks $X$ and $X+Z$. The first main (classic) results were obtained in \cite{RS70,SSu03,S65}. These results and some additional results can also be seen in \cite{R04,Z21} and in the classic books on stochastic comparisons  \cite{DDGK05,MS02,SS07}. Recently, Guan et al. \cite{GHW24} have obtained some improvements of these classic results. The details are given in the following section. 

The purpose of the present paper is to use the recent results obtained in \cite{GHW24} to get stochastic comparisons between $X$ and $X+Z$. We use a copula approach which allows us to compare $X$ with $X+Z$ under some conditions for the copula $C$ of $(X,Z)$. The conditions obtained for $C$ are weak dependence properties and they allow us to get distribution-free comparison results, that is, we can change the marginal distributions in a fixed dependence structure. To do so we need to define new weak  positive dependence properties based on the underlying copula. We also need some symmetry (or asymmetry) properties for the distribution of  $Z$.

The rest of the paper is scheduled as follows. In the following section we establish the notation and we provide the basic results needed for the present paper.  The main results are placed in Section 3 where we also define the new dependence properties.  Some illustrative examples are provided in Section 4. Finally, in Section 5,  we state the main conclusions about the obtained results and some open questions for future research.     

\section{Preliminary results}

Let $\mathcal{L}_1$ be the set of integrable random variables over an atomless probability space $(\Omega, \mathcal{F}, \Pr)$. 
Let $X\in \mathcal{L}_1 $ be a random variable representing a random payoff or wealth and let $F(t)=\Pr(X\leq t)$  be its cumulative distribution function (CDF).  Its survival (or reliability) function is $\bar F(t)=1-F(t)=\Pr(X>t)$ and if $F$ is absolutely continuous, its probability density function (PDF) is defined by $f(t)=F'(t)$ wherever these derivatives exist.

The expected value (mean) of $X$ can be obtained as
\begin{equation}\label{mean}
	\mu=E(X)=\int_{-\infty}^{\infty} xdF(x)= \int_{0}^{\infty} \bar F(x)dx-\int_{-\infty}^{0} F(x)dx=\mu_+-\mu_-,
\end{equation}
where $\mu_+=\int_{0}^{\infty} \bar F(x)dx$ and  $\mu_-=\int_{-\infty}^{0}  F(x)dx$. Note that $\mu$ exists if and only if 
$\mu_+$ or $\mu_-$ are finite and $\mu$ is finite if both are finite. If $X$ is nonnegative we just need the first integral in \eqref{mean}, that is $\mu=\mu_+$. 

We shall use the following stochastic orders. Their main properties and applications can be seen in  \cite{DDGK05,MS02,SS07}. Throughout the paper, `increasing' means `non-decreasing' and  `decreasing' means `non-increasing'.  Whenever we use a derivative, an expectation or a conditional distribution, we are tacitly assuming that they exist.  

If $X,Y\in \mathcal{L}_1 $, then we say that:
\begin{itemize}
	\item
	$X$ is less  than $Y$ in the stochastic order (or in the first-order stochastic dominance), shortly written as $X\leq_{ST}Y$,  if $\bar F_X\leq \bar F_Y$ or, equivalently, if $E(u(X))\leq E(u(Y))$ for all increasing functions $u$ such that these expectations exist. Then $X=_{ST}Y$ represents equality in law (distribution).
	
	\item $X$ is less  than $Y$ in the increasing concave  order (or in the second-order stochastic dominance), shortly written as $X\leq_{ICV}Y$,  if $E(u(X))\leq E(u(Y))$ for all increasing concave functions $u$ such that these expectations exist or, equivalently, if 
	$$\int_{-\infty}^t F_X(x)dx\geq \int_{-\infty}^t F_Y(x)dx\ \text{ for all }t.$$

	\item $X$ is less than $Y$ in the increasing convex  order (or in the stop-loss order), shortly written as $X\leq_{ICX}Y$,  if $E(u(X))\leq E(u(Y))$ for all increasing convex functions $u$ such that these expectations exist or, equivalently, if 
	$$\int_t^{\infty} \bar F_X(x)dx\leq \int_t^{\infty} \bar F_Y(x)dx\ \text{ for all }t.$$
	The function  $\pi(t)=\int_t^{\infty} \bar F(x)dx$ is called \textit{stop-loss  function} (see e.g.  \cite{MS02}, p.\ 19). Hence, $X\leq_{ICX}Y$ holds if and only if $\pi_X\leq \pi_Y$.  It is also equivalent to $-X\geq_{ICV}-Y$.

	\item $X$ is less  than $Y$  in the convex  order, shortly written as $X\leq_{CX}Y$,  if $E(u(X))\leq E(u(Y))$ for all convex functions $u$ such that these expectations exists or, equivalently, if $\pi_X\leq \pi_Y$  and $\lim_{t\to -\infty} \pi_X(t)-\pi_Y(t)=0$. It implies $E(X)=E(Y)$  and $Var (X)\leq Var (Y)$. It is also  equivalent to $X\leq_{ICX}Y$ and $E(X)=E(Y)$ (see \cite{MS02}, p.\ 19) and to $X\geq_{ICV}Y$ and $E(X)=E(Y)$.

	\item $X$ is less  than $Y$  in the convolution  order, shortly written as $X\leq_{CONV}Y$  if $Y=_{ST}X+Z$ where $Z$ is a nonnegative random variable independent of $X$  (see  \cite{SS07}, p.\ 70 and \cite{SSu03}).  
	
\end{itemize}

The relationships between these orders are summarized as follows: 
\[
\begin{tabular}{ c c c c c c c}
	$X\leq_{CONV}Y$ & $\Rightarrow$  &$X\leq_{ST}Y$ &  $\Rightarrow$ & $X\leq_{ICX}Y$&$\Leftarrow$ &$X\leq_{CX}Y$\\%
	&&	$\Downarrow$    
	&               & $\Downarrow$    &               &$\Downarrow$
	\\%
	$X\geq_{CX}Y$&$\Rightarrow$&$X\leq_{ICV}Y$ &  $\Rightarrow$ & $E(X)\leq E(Y)$&&$E(X)=E(Y)$%
\end{tabular}
\]

\quad

Now we are ready to state the classic results on this topic. They were  obtained in \cite{RS70,S65} and  can be summarized as follows. They can also be seen in Theorem 4.A.5 in \cite{SS07}, p.\ 183.

\begin{proposition}\label{ICV1}
	For $X,Y\in \mathcal{L}_1 $, $X\geq_{ICV}Y$ holds if and only if $Y=_{ST}W+Z$ for  $W,Z\in \mathcal{L}_1 $ such that $W=_{ST}X$ and  $E(Z|W)\leq 0$.
\end{proposition}  

\begin{proposition}
	For $X,Y\in \mathcal{L}_1 $,	$X\leq_{ICX}Y$ holds if and only if $Y=_{ST}W+Z$ for  $W,Z\in \mathcal{L}_1$ such that $W=_{ST}X$ and $E(Z|W)\geq 0$. 
\end{proposition}



An analogous result can be stated for the stochastic order from Theorem 1.A.1 in \cite{SS07}, p. 5.

\begin{proposition}\label{prop2.4}
	For $X,Y\in \mathcal{L}_1 $, $X\leq_{ST}Y$ holds if and only if  $Y=_{ST}W+Z$ for  $W,Z\in \mathcal{L}_1$ such that $W=_{ST}X$ and such that $Z\geq 0$. 
\end{proposition}  

Note that the random variable $Z$ used in the preceding propositions may depend on $W$. The convolution order is already defined with a property of this type where $Z$ is assumed to be independent of $W$.  So it implies the ST order. Analogously, from these properties, it is clear that   the ST order implies both the ICV and  ICX orders (a well known property that can also be deduced from the definitions).

Now we state the improvements of these results obtained recently in Guan et al. \cite{GHW24}. They will be the basic tools for the results  included in  the present paper.

\begin{proposition}[Theorem 1 in \cite{GHW24}]\label{G1}
	For $X,Y\in \mathcal{L}_1 $, 	the order $X\geq_{ICV}Y$ holds if and only if $Y=_{ST}W+Z$ for  $W,Z\in \mathcal{L}_1$ such that $W=_{ST}X$ and  $E(Z|X\leq x)\leq 0$ for all values $x$ for which these expectations exist.  
\end{proposition}  

Note that the condition ``$E(Z|X\leq x)\leq 0$ for all $x$'' implies $E(Z)\leq 0$ and that it is weaker than the condition ``$E(Z|X)\leq 0$'' assumed in Proposition \ref{ICV1} (see \cite{GHW24}).

\begin{proposition}[Corollary 4 in \cite{GHW24}]\label{G2}
	
	For $X,Y\in \mathcal{L}_1 $, 	the order $X\leq_{ICX}Y$ holds if and only if $Y=_{ST}W+Z$ for  $W,Z\in \mathcal{L}_1$ such that $W=_{ST}X$ and  $E(Z|X> x)\geq 0$ for all values $x$ for which these expectations exist.  
\end{proposition}

\begin{proposition}[Corollary 5 in \cite{GHW24}]\label{G3}
	For $X,Y\in \mathcal{L}_1 $, the following conditions are equivalent: 
	\begin{itemize}
		\item[(i)] $X\leq_{CX}Y$;
		\item[(ii)] $Y=_{ST}W+Z$ for  $W,Z\in \mathcal{L}_1$ such that $W=_{ST}X$ and  $E(Z|W)=0$;
		\item[(iii)] $Y=_{ST}W+Z$ for  $W,Z\in \mathcal{L}_1$ such that $W=_{ST}X$,  $E(Z)=0$ and $E(Z|W\leq x)\leq 0$ for all values $x$ for which these expectations exist;
		\item[(iv)] $Y=_{ST}W+Z$ for  $W,Z\in \mathcal{L}_1$ such that $W=_{ST}X$,  $E(Z)=0$ and $E(Z|W> x)\geq 0$ for all values $x$ for which these expectations exist.
	\end{itemize}   
\end{proposition}





\section{Main results}

Now we want to study conditions to get comparison results between the risks $X$ and $X+Z$. 
To this end we need the following dependence concept defined in \cite{BL09}, p.\ 117. 


\begin{definition}
	For $X,Z\in \mathcal{L}_1 $, the random vector $(X,Z)$ is Positive  Quadrant Dependent in Expectation for $Z$ given $X$, shortly written as   $PQDE(Z|X)$,  if 
	\begin{equation} \label{PQDE}
		E(Z|X> x)\geq E(Z)\   \text{ for all } x\in\mathbb{R}.
	\end{equation} 
\end{definition}

The negative dependence notion $NQDE(Z|X)$ is defined by reverting the inequality in \eqref{PQDE}.
Note that \eqref{PQDE} holds if and only if
\begin{equation} \label{PQDEbis}
	E(Z|X\leq x)\leq E(Z)\   \text{ for all } x\in\mathbb{R}.
\end{equation}
In \cite{BL09}, p.\ 118, it is proved that $PQDE(Z|X)$ implies $Cov(X,Z)\geq 0$. 

The PQDE is a weaker version of the classic PQD property defined as 
\begin{equation*} \label{PQD1}
	(Z|X> x)\geq_{ST} Z\   \text{ for all } x\in\mathbb{R}
\end{equation*} 
The PQD  property can also be written as 
\begin{equation}\label{PQD}
	\Pr(X\leq x,Z\leq z)\geq  \Pr(X\leq x)\Pr(Z\leq z)\ \text{ for all } x,z\in\mathbb{R}
\end{equation}
or as 
$$\Pr(X>x,Z>z)\geq  \Pr(X>x)\Pr(Z>z)\ \text{ for all } x,z\in\mathbb{R} $$
(see \cite{BL09}, p.\ 108).
In \cite{NPS21} it is proved that several positive  dependence notions (defined in a similar way) imply the PQDE or the PQD properties.

Now, by using these dependence concepts, we can state the following result.

\begin{proposition}\label{ICX}
	For $X,Z\in \mathcal{L}_1 $, the order  	$X\leq_{ICX}X+Z$ holds  if $E(Z)\geq 0$ and \eqref{PQDE} holds. 
\end{proposition}



The proof is immediate since if  $E(Z)\geq 0$ and $(X,Z)$ satisfies \eqref{PQDE}, then  $$E(Z|X> x)\geq E(Z)\geq 0$$ holds and so, from Proposition \ref{G2},    we get $X\leq_{ICX}X+Z$.


Note that the condition $E(Z)\geq 0$ is necessary to get the order $X\leq_{ICX}X+Z$ (since the ICX order  implies the order in means). However, we do not know if the $PQDE(Z|X)$ is a necessary property for that ordering since it is not equivalent to the property assumed in Proposition \ref{G2}. 
The advantage of the new statement is that the 
$PQDE(Z|X)$ property is sometimes easy to check than that property. For example, it holds if $(X,Z)$ is PQD. We have  similar results for the ICV and CX orders.

\begin{proposition}
	For $X,Z\in \mathcal{L}_1 $,   $X\geq_{ICV}X+Z$ holds  if  $E(Z)\leq 0$ and \eqref{PQDE} holds. 
\end{proposition}

The proof is analogous to that of  Proposition \ref{ICX} by using now  \eqref{PQDEbis} and Proposition \ref{G1}.

\begin{proposition}
	For $X,Y\in \mathcal{L}_1 $, the order $X\leq_{CX}Y$ holds if and only if $Y=_{ST}W+Z$ for  $W,Z\in \mathcal{L}_1 $ such that $W=_{ST}X$ and  $E(Z)= 0$ and \eqref{PQDE} holds for $(W,Z)$. 
\end{proposition}  
\begin{proof} \quad 
	
	The proof of the `if' part  is immediate from Proposition \ref{G3}, $(iv)$. 
	
	To prove the `only if' part we also use Proposition \ref{G3}, $(iv)$ and we  note  that if $X\leq_{CX}W+Z$ holds for $W=_{ST}X$, then $E(X)=E(W+Z)=E(W)+E(Z)$ and so $E(Z)=0$ since $E(W)=E(X)$. We also note that if $E(Z|W>x)\geq 0$ for all $x$, then $(W,Z)$ is $PQDE(Z|W)$.  
\end{proof}


These results can also be stated by using a copula approach. Thus, from the celebrated Sklar's theorem, we know that the joint distribution function of $(X,Z)$ can be written as 
\begin{equation*}
	\Pr(X\leq x,Z\leq z)= C(F(x),G(z))\ \text{ for all } x,z\in\mathbb{R},
\end{equation*}
where $C$ is a copula function and $F(x)=\Pr(X\leq x)$ and $G(z)=\Pr(Z\leq z)$ are the marginal CDF.  Moreover, if $F$ and $G$ are continuous, then the copula is unique.  Actually, $C$ is the joint CDF of $(F(X),G(Z))$. For basic properties and examples of copulas see the classic book Nelsen  \cite{N06}.

Analogously, the joint survival function can be written as  
\begin{equation}
	\Pr(X> x,Z> z)= \widehat C(\bar F(x),\bar G(z))\ \text{ for all } x,z\in\mathbb{R},
\end{equation}
where $\widehat C$ is a copula function (called \textit{survival copula}) and $\bar F(x)=\Pr(X> x)$ and $\bar G(z)=\Pr(Z> z)$ are the marginal survival functions.
It is well known that 
\begin{equation}\label{hatC}
	\widehat C(u,v) = u + v -1 + C(1-u, 1-v)  \text{ for all } u,v\in [0,1].
\end{equation}

Note that the PDQ property holds for all $F,G$ if $\widehat C$ (or $C$) is PQD in the sense of \eqref{PQD}, that is, it satisfies
\begin{equation}\label{C-PQD}
	\widehat C(u,v)\geq uv  \text{ for all } u,v\in [0,1].
\end{equation}
The reverse property is true for continuous marginals, that is, if $F$ and $G$ are continuous and $(X,Z)$ is PQD, then \eqref{C-PQD} holds for $\widehat C$ (or $C$).  
To the best of our knowledge there are not in the literature a similar copula property for the PQDE notion defined in \eqref{PQDE}. From \eqref{mean}, it is easy to see that  \eqref{PQDE} is equivalent to 
$$\int_{-\infty}^{\infty} \left[\widehat C(\bar F(x),\bar G(z))-\bar F(x)\bar G(z) \right]dz\geq 0\ \text{ for all } x.$$

The main advantage of the copula approach  is that if 
\eqref{C-PQD} holds, then \eqref{PQDE} holds for all $F,G$. Thus, from Proposition \ref{ICX}, if $\widehat C$ satisfies 
\eqref{C-PQD}, then $X\leq_{ICX}X+Z$ holds for all distribution functions $F,G$ such that  $E(Z)\geq 0$.  
Similar results holds for the ICV and CX orders. These are distribution-free comparison results. 

There are many examples of copulas satisfying the PQDE (or the stronger PQD) property. For example,  the Archimedean copulas 
$$C(u,v)= \psi^{-1}(\psi(u)+\psi(v))$$
with strict generator $\psi$ are PQD if and only if  $-\log \psi^{-1}$ is subadditive in $(0,\infty)$ (see Exercise 5.34 in \cite{N06}, p.\ 205). This property is equivalent to 
$$\psi(u)+\psi(v)\leq \psi(uv)\ \text{ for all }u,v\in (0,1).$$

The $PQDE(Z|X)$ property in \eqref{PQDE} holds if $(X,Z)$ are positively stop-loss dependent ($PSLD)(Z|X)$) as defined in \cite{DDGK05}, p.\ 270 (see also the property $RTI^0_{ICX}(Z|X)$ in \cite{NPS21}) which implies
$$Z\leq_{ICX}(Z|X>x)\ \text{ for all } x \text{ such that }\bar F(x)>0.$$ 
Hence, from Proposition \ref{ICX}, the $PSLD(Z|X)$ property jointly with $E(Z)\geq 0$ implies $X\leq_{ICX}X+Z$ (since the ICX order, implies the order in means and so the $PQDE(Z|X)$ property holds).

\quad

Now we want to obtain similar results by using weaker conditions than \eqref{C-PQD} for the survival copula $\widehat{C}$. 
For the first result we need the following symmetry  property for $Z$. 

\begin{definition}
	A random variable $Z$ with CDF $G$ is symmetric around zero if $$G(-z)=1-G(z)\\ \text{ for all }z\text{ such that }G(z-)=G(z).$$  	A random variable $Z$ is symmetric around $a\in\mathbb{R}$ if $Z-a$ is symmetric around zero.
\end{definition}

From \eqref{mean}, if $Z$ is symmetric around zero and has a finite mean, then $E(Z)=0$. If $Z$  has an absolutely continuous distribution, then it is symmetric around zero if and only if  it has a PDF $g$ such that $g(z)=g(-z)$ for all $z$. Now we can state the following result.

\begin{proposition}
	If $X,Z\in\mathcal{L}_1$, $Z$ is symmetric around zero and  the survival copula  $\widehat C$ of $(X,Z)$  satisfies 
	\begin{equation}\label{C-PQDE}%
		\widehat{C}(u,v)+\widehat{C}(u,1-v) \geq u  \text{ for all } u,v\in [0,1],	
	\end{equation}	 
	then $X\leq_{CX}X+Z$. 
\end{proposition}
\begin{proof}
	As $Z$ is symmetric around zero, then  $E(Z)=0$.  Hence we just need to prove the $PQDE(Z|X)$ property 
	$E(Z|X> x)\geq E(Z)=0$   for all  $x$ such that $\bar F(x)>0$. The survival function of $(Z|X> x)$ can be obtained as 
	$$\Pr(Z>z|X> x)=\frac{\Pr(X>x, Z>z)}{\Pr(X>x)}=\frac{\widehat C(\bar F(x),\bar G(z))}{\bar F(x)}$$
	for all $x$ such that $\bar F(x)>0$. Hence, from \eqref{mean},
	\begin{align*}
		E(Z|X> x)&=\int_0^\infty \Pr(Z>z|X> x) dz - \int_{-\infty}^0 (1-\Pr(Z>z|X> x)) dz\\ 
		&=\int_0^\infty \frac{\widehat C(\bar F(x),\bar G(z))}{\bar F(x)} dz - \int_{-\infty}^0  \left(1-\frac{\widehat C(\bar F(x),\bar G(z))}{\bar F(x)}\right) dz\\ 
		&=\int_0^\infty \frac{\widehat C(\bar F(x),\bar G(z))}{\bar F(x)} dz - \int_0^{\infty} \left(1-\frac{\widehat C(\bar F(x),\bar G(-z))}{\bar F(x)}\right) dz\\
		&=\int_0^\infty \frac{\widehat C(\bar F(x),\bar G(z))+\widehat C(\bar F(x),\bar G(-z))-\bar F(x)}{\bar F(x)} dz \\
		&=\int_0^\infty \frac{\widehat C(\bar F(x),\bar G(z))+\widehat C(\bar F(x),1-\bar G(z))-\bar F(x)}{\bar F(x)} dz,
	\end{align*}
	where the last equality is obtained from the symmetry property for $Z$. Hence, if \eqref{C-PQDE} holds, then  
	$E(Z|X> x)\geq 0$ and the stated CX order holds from Proposition \ref{G3}. 
\end{proof}

Of course, the product copula  satisfies \eqref{C-PQDE} and so $X\geq_{CX}X+Z$ holds for independent $X$ and $Z$ (convolutions) with $Z$ having symmetric distributions around zero (e.g. for Gaussian distribution with zero mean). 
However, the reverse ordering   $X\geq_{CX}X+Z$ does not necessarily  hold if the inequality in \eqref{C-PQDE} is reversed.  Note that  $X\geq_{CX}X+Z$, implies $Var(X)\geq Var(X+Z)$ and, if $X$ and $Z$ are independent, then $$Var(X)\geq Var(X+Z)=Var(X)+Var(Z),$$ which leads to $Var(Z)=0$. Hence, if we assume  $E(Z)=0$, we get $Z=0$ (almost surely) and $X+Z=X$ (almost surely).


To get similar results for the ICX and ICV orders  we need the following asymmetry (or skew)  property for $Z$. 

\begin{definition}
A random variable  $Z$ with CDF $G$  is right (left) skewed 
if 
	\begin{equation}\label{RS}
		G(-z)\leq 1-G(z)\ (\geq)\quad  \text{ for all } z\geq 0 \text{ such that } G(z-)=G(z).
	\end{equation} 
\end{definition}

Clearly, it means that the right tails have greater (smaller) weights (probability)  than the left tails.  
Note that if $Z$ has a finite mean and it is right (left) skewed, then from \eqref{mean}, $E(Z)\geq 0$ ($\leq$). If $Z$ has an absolutely continuous distribution with PDF $g$, then \eqref{RS} holds if 
$$g(z)\geq g(-z)\ \text{for all }z\geq 0.$$

Now we can state the following results for the ICX and ICV orders.

\begin{proposition}\label{prop3.7}
	If $Z$ is right skewed  and  the survival copula  $\widehat C$ of $(X,Z)$  satisfies \eqref{C-PQDE}, 	 
	then $X\leq_{ICX}X+Z$.
\end{proposition}
\begin{proof}From the preceding proof we get
	\begin{align*}
		E(Z|X> x)&=\int_0^\infty \frac{\widehat C(\bar F(x),\bar G(z))+\widehat C(\bar F(x),\bar G(-z))-\bar F(x)}{\bar F(x)} dz \\
		&\geq \int_0^\infty \frac{\widehat C(\bar F(x),1-\bar G(-z))+\widehat C(\bar F(x),\bar G(-z))-\bar F(x)}{\bar F(x)} dz,
	\end{align*}
	where the last equality is obtained from the right skewed property \eqref{RS} of $Z$ that can be written as $\bar G(z)\geq 1-\bar G(-z)$ for all $z\geq 0$.  Hence, if \eqref{C-PQDE} holds, then  
	$E(Z|X> x)\geq 0$ and the stated ICX order holds  from Proposition \ref{G2}.
\end{proof}

\begin{proposition}\label{prop3.9}
	If $Z$ is left skewed  and  the copula  $C$ of $(X,Z)$  satisfies 
	\begin{equation}\label{CC-PQDE}%
		{C}(u,v)+{C}(u,1-v) \geq u  \text{ for all } u,v\in [0,1],	
	\end{equation}	 
	then $X\geq_{ICV}X+Z$. 
\end{proposition}
\begin{proof}
	The conditional distribution function of $(Z|X> x)$ is 
	$$\Pr(Z\leq z|X\leq  x)=\frac{\Pr(X\leq x, Z\leq z)}{\Pr(X\leq x)}=\frac {C(F(x), G(z))}{ F(x)}$$
	for all $x$ such that $F(x)>0$.	Hence, from \eqref{mean}, 
	\begin{align*}
		E(Z|X\leq x)&=\int_0^\infty (1-\Pr(Z\leq z|X\leq  x)) dz - \int_{-\infty}^0 \Pr(Z\leq z|X\leq x) dz\\ 
		&=\int_0^\infty \left(1-\frac{C( F(x),G(z))}{ F(x)}\right)dz - \int_{-\infty}^0 \frac{ C( F(x), G(z))}{ F(x)}   dz\\ 
		&=\int_0^\infty \left(1-\frac{C( F(x),G(z))}{ F(x)}\right)dz - \int_0^{\infty} \frac{ C( F(x), G(-z))}{ F(x)}   dz\\ 
		&=\int_0^\infty \frac{F(x)-C( F(x), G(z))- C( F(x), G(-z))}{ F(x)} dz \\
		&\leq \int_0^\infty \frac{F(x)-C( F(x), G(z))- C( F(x), 1-G(z))}{ F(x)} dz,
	\end{align*}
	where the last equality is obtained from the left skewed property that can be written as $1- G(z)\leq  G(-z)$ for all $z\geq 0$.  Hence, if \eqref{CC-PQDE} holds, then  
	$E(Z|X\leq  x)\leq 0$ and the stated ICV order holds for all $F,G$ from Proposition \ref{G1}.
\end{proof}

It is easy to see from \eqref{hatC} that the properties  \eqref{C-PQDE} and 
\eqref{CC-PQDE} are equivalent. Hence, the preceding proof can also be obtained by replacing $X$ and $Z$ with $-X$ and $-Z$ in the result for the ICX order. 
Moreover,  this weak dependence property can be defined (as the PQD property) equivalently by using the copula or the survival copula. Thus we propose the following definition.

\begin{definition}
	We say that a random vector $(X,Z)$ with a copula $C$ is weakly Positively (Negatively) Quadrant Dependent for $Z$ given $X$, shortly written as $wPQD(Z|X)$ ($wNQD(Z|X)$), if $C$ satisfies 
	\begin{equation}\label{C-PQDEbis}%
		C(u,v)+C(u,1-v) \geq u \ (\leq)\ \text{ for all } u,v\in [0,1].	
	\end{equation}
\end{definition}

Note that the new dependence concepts only depend on the copula $C$. 
The concepts $wPQD(X|Z)$ and $wNQD(X|Z)$ for the other conditioning are defined in a similar way. They are not equivalent to the preceding ones, see Example \ref{ex2}. If both $wPQD(Z|X)$ and $wPQD(X|Z)$ hold, then we just write $wPQD(X,Z)$.  

It is easy to see that $PQD(X,Z)$ implies that both $wPQD(Z|X)$ and $wPQD(X|Z)$ hold. However, the reverse property does not hold, that is, $wPQD(Z|X)$ does not imply $PQD(X,Z)$, see Example \ref{ex2}.
Even more, in that example, we prove that  $wPQD(Z|X)$ does not imply $Cov(X,Z)\geq 0$. Hence, the $wPQD(Z|X)$ property may hold even if $PQDE(Z|X)$ does not hold. Also note that the wPQD concept can be applied to bivariate random vectors for which the means do not exist.    However, to apply the PQDE property we need to assume the existence of the means in \eqref{PQDE}. 

The Gini measure of association $\gamma_C$ is defined as 
$$\gamma_C=4\int_0^1 \left[ C(u,u)+C(u,1-u)-u\right]du,$$
see \cite{N06}, p.\ 181. It is a concordance measure. Clearly,  $wPQD(Z|X)$ (resp. $wNQD(Z|X)$) implies $\gamma_C\geq 0$ ($\leq$). The same holds for $wPQD(X|Z)$ (resp. $wNQD(X|Z)$).

{The Spearman's rho correlation coefficient $\rho_C$ can be computed as
	$$\rho_C=-3+12\int_0^1 \int_0^1  C(u,v)dvdu=12\int_0^1 \int_0^1 \left[ C(u,v)-uv\right]dvdu.$$
	So $\rho_C/12$ can be seen as a ``a measure of average quadrant dependence'' (see \cite{N06}, p.\ 189). If $(U,V)$ has $C$ as CDF, then $$\rho_C=12E(UV)-3$$ (see \cite{N06}, p.\ 167).} Of course, for the product copula $\Pi(u,v)=uv$, we get $\rho_{\Pi}=0$. For the Fréchet-Hoeffding upper and lower bounds $M(u,v)=\min(u,v)$ and $W(u,v)=\max(u+v-1,0)$ (see \cite{N06}, p.\ 11), we have $\rho_M=1$ and $\rho_W=-1$. Clearly, if two copulas are ordered $C_1\leq C_2$, then $\rho_{C_1}\leq \rho_{C_2}$.

We could define a measure similar to $\gamma_C$ from the expression used to define the $wPQD(Z|X)$ property but the next proposition shows that it leads to the Spearman's rho coefficient. Hence,  $\rho_C$  can also be seen as a generalization of the Gini measure of association by using this new dependence measure.

\begin{proposition}
	If the generalized Gini measure of association of a random vector with copula $C$ is defined as
	$$\delta_C=6\int_0^1 \int_0^1 \left[ C(u,v)+C(u,1-v)-u\right]dvdu,$$
	then $\delta_C=\rho_C$.
\end{proposition}
\begin{proof}
	First we note that if $(U,V)$ is a random vector with  CDF $C$, then $\rho_C=12E(UV)-3$	(see \cite{N06}, p.\ 167). We also note that $D(u,v)=u-C(u,1-v)$ is the CDF (copula) of $(U,1-V)$. Hence 
	$$\delta_C=\frac 1 2 (\rho_C-\rho_D)=6E(UV)-6E(U(1-V))=12E(UV)-6E(U)=\rho_C$$
	and the proof is completed. \end{proof}

\quad

The preceding proposition proves that the delta measures are symmetric, that is, 
$$\int_0^1 \int_0^1 \left[ C(u,v)+C(u,1-v)-u\right]dvdu=\int_0^1 \int_0^1 \left[ C(u,v)+C(1-u,v)-v\right]dvdu.$$
Hence, if $wPQD(Z|X)$ or $wPQD(X|Z)$ holds, then $\rho_C\geq 0$. A similar property holds for the negative dependence properties.

Note that  $wPQD(Z|X)$ holds if and only if $C\geq D$, where 
$$D(u,v)=u- C(u,1-v)$$ is the copula of $(U,1-V)$ and $(U,V)$ is a random vector with CDF $C$.  Also note that if $wPQD(Z|X)$ and $wNQD(Z|X)$ hold, that is,
$$C(u,v)+C(u,1-v)-u=0\ \text{ for all } u,v\in[0,1],$$
then $(U,V)=_{ST} (U,1-V)$ (or $C=D$). In Examples \ref{ex1} and \ref{ex2}, we provide  copulas satisfying this property that are not PQD. Hence, we show that $wPQD(Z|X)$ does not implies $PQD(X,Z)$, that is, it is a strictly weaker positive dependence notion. Example \ref{ex2} also proves that $wPQD(Z|X)$ does not implies $PQDE(Z|X)$. For the covariance, we have the following results.

\begin{proposition}
	If $(X,Z)$ is $wPQD(Z|X)$ (resp. $wNQD(Z|X)$) and $Z$ is symmetric around $a\in\mathbb{R}$, then $Cov(X,Z)\geq 0$ ($Cov(X,Z)\leq 0$). 
\end{proposition}

\begin{proof}
	As $Cov(X,Z)=Cov(X,Z-a)$ we can assume that $a=0$.
	If  $Z$ is symmetric around zero and has a CDF $G$, then $1-G(z)=G(-z)$ for all $z$. From \cite{N06}, p.\ 150, the covariance can be computed as 
	\begin{align*}
		Cov(X,Z)
		&=\int_{-\infty}^{\infty}\int_{-\infty}^{\infty}
		\left[ C(F(x),G(z))-F(x)G(z) \right] dzdx\\
		&=\int_{-\infty}^{\infty}\int_{-\infty}^{0}
		\left[ C(F(x),G(z))-F(x)G(z) \right]dz dx\\
		&\quad +\int_{-\infty}^{\infty}\int_0^{\infty}
		\left[ C(F(x),G(z))-F(x)G(z) \right] dz dx\\
		&=\int_{-\infty}^{\infty}\int_0^{\infty}
		\left[ C(F(x),G(-z))-F(x)G(-z) \right]dz dx\\
		&\quad +\int_{-\infty}^{\infty}\int_0^{\infty}
		\left[ C(F(x),G(z))-F(x)G(z) \right] dz dx\\
		&=\int_{-\infty}^{\infty}\int_0^{\infty}
		\left[ C(F(x),1-G(z))+C(F(x),G(z))-F(x) \right]dz dx,
	\end{align*}
	where the last equality is obtained from the symmetry property of $Z$. Finally, by using \eqref{C-PQDEbis}, we get 
	$Cov(X,Z)\geq 0$ (whenever this covariance exists). The proof for $wNQD(Z|X)$ is analogous.
\end{proof}

Example \ref{ex2} proves that 
$wPQD(Z|X)$ does not imply $Cov(X,Z)\geq 0$ when we drop out the condition about the symmetry of $Z$. As an immediate consequence we get the following result.

\begin{proposition}\label{PropCov}
	If $Z$ is symmetric around $a\in\mathbb{R}$ and  $(X,Z)$ has the copula $C$ and  is both $wPQD(Z|X)$ and  $wNQD(Z|X)$, then $\gamma_C=\rho_C=Cov(X,Z)= 0$.
\end{proposition}



Let us obtain now similar results by using copula representations for the PDF in the absolutely continuous case. The first result is again for the convex order and  a random variable $Z$  symmetric around zero. We use the notation $\partial _i H$ for the partial derivative of the function $H$ with respect to its $i$th variable. When we write $\partial _i H(a,b)$ we mean that this partial derivative is evaluated at the point $(a,b)$.

\begin{proposition}\label{prop3.14}
	If $Z\in\mathcal{L}_1$ is symmetric around zero and $(X,Z)$ is absolutely continuous and its survival copula  satisfies 
	\begin{equation}\label{D1}%
		\partial_2  \widehat C(u,v) \geq \partial_2  \widehat C(u,1-v)  \text{ for all } u\in (0,1) \text{ and all }v\in (0,1/2),	
	\end{equation}	 
	then $X\leq_{CX}X+Z$. 
\end{proposition}
\begin{proof}	
	First,  we recall the representation for the survival function of  $(Z|X>x)$
	$$\bar G(z|x):=\Pr(Z>z|X> x)=\frac{\Pr(X>x, Z>z)}{\Pr(X>x)}=\frac{\widehat C(\bar F(x),\bar G(z))}{\bar F(x)}$$
	for all $z$ and all $x$ such that $\bar F(x)>0$. 
	If $(X,Z)$ has an absolutely continuous distribution, the  PDF $g(z|x)$ of  $(Z|X>x)$ is
	$$g(z|x)=-\frac{\partial}{\partial z}\bar G(z|x)=\frac{\partial_2 \widehat C(\bar F(x),\bar G(z))}{\bar F(x)}g(z),$$
	where $g=-\bar G'$ is the PDF of $Z$ and $\partial_{2}\widehat C$ represents the partial derivative of $\widehat C$ with respect to its second variable. Hence 
	\begin{align*}
		E(Z|X> x)&
		=\int_{-\infty}^\infty zg(z|x) dz\\ 
		&=\int_{-\infty}^\infty \frac{\partial_2  \widehat C(\bar F(x),\bar G(z))}{\bar F(x)}zg(z)  dz\\
		&=\int_0^\infty \frac{\partial_2  \widehat C(\bar F(x),\bar G(z))}{\bar F(x)}zg(z)   dz + \int_{-\infty}^0  \frac{\partial_2  \widehat C(\bar F(x),\bar G(z))}{\bar F(x)}zg(z)  dz\\
		&=\int_0^\infty \frac{\partial_2  \widehat C(\bar F(x),\bar G(z))}{\bar F(x)}zg(z)   dz - \int_0^{\infty}  \frac{\partial_2  \widehat C(\bar F(x),\bar G(-z))}{\bar F(x)}zg(-z)  dz.
	\end{align*}
	Now we use the symmetry property for $Z$ that, in the absolutely continuous case, leads to $\bar G(0)=1/2$, $g(-z)=g(z)$ and $\bar G(-z)=1-\bar G(z)$ for all $z$. Then 
	\begin{align*}
		E(Z|X> x)
		&=\int_0^\infty \frac{\partial_2  \widehat C(\bar F(x),\bar G(z))-\partial_2  \widehat C(\bar F(x),1-\bar G(z))}{\bar F(x)}zg(z)   dz
	\end{align*}
	which is nonnegative if \eqref{D1} holds since $\bar G(z)\leq \bar G(0)=1/2$ for $z\geq 0$. Hence $X\leq_{CX}X+Z$ holds from $E(Z)=0$ and Proposition \ref{G3}.
\end{proof}

Note that we just need the absolutely continuous assumption for the random variable $(Z|X>x)$. Also note that, from the proof,  $E(Z|X> x)\geq 0$ holds if, and only if 
\begin{equation}\label{eq3.12}
	E(Z\partial_2  \widehat C(u,\bar G(Z)))\geq 0\  \text{ for all }u\in (0,1).
\end{equation}  So, in the preceding proposition, \eqref{D1} can be replaced with this property. However, in general, it is not easy to check \eqref{eq3.12}since it depends on both $\widehat C$ and $\bar G$.

There are many copulas satisfying the property \eqref{D1}. For example, it holds for the product copula (independence case)  and for the  Farlie-Gumbel-Morgenstern copula
$$\widehat C(u,v)=uv+\theta uv(1-u)(1-v)$$
when  $\theta \geq 0$. 
It is easy to see from \eqref{hatC} that \eqref{D1} is equivalent to 	
\begin{equation}\label{D1bis}%
	\partial_2  C(u,v) \geq   \partial_2 C(u,1-v)  \text{ for all } u\in (0,1) \text{ and }v\in (0,1/2).	
\end{equation}
Thus we can state the following definition.
\begin{definition}
	We say that a random vector $(X,Z)$ with a copula $C$ is strongly  Positively (Negatively) Quadrant Dependent for $Z$ given $X$, shortly written as $sPQD(Z|X)$ ($sNQD(Z|X)$), if $C$ satisfies \eqref{D1bis} ($C$ satisfies  \eqref{D1bis} with the reverse inequality).
\end{definition}

Note that condition $sPQD(Z|X)$ holds when $\partial_2   C(u,v)$ is decreasing for $v\in(0,1)$, that is,  $ C(u,v)$ is concave in $v$. This property is equivalent to the stochastically increasing property  $SI_{ST}(Z|X)$ defined as
$$(Z|X=x_1)\leq_{ST} (Z|X=x_2)\ \text{ for all }\ x_1<x_2.$$ 
The $SI_{ST}(Z|X)$ property is a strong positive dependence property that implies the PQD property (see e.g.  \cite{NPS21}).
However, \eqref{D1} is a weaker condition than $SI_{ST}(Z|X)$ and we do not know if the $PQD(X,Z)$ property implies $sPQD(Z|X)$. 
Example \ref{ex2} proves that $sPQD(Z|X)$ does not imply $PQD(X,Z)$. Moreover,  we can prove the following properties.

\begin{proposition}\label{propnew1}
	If $C$ is an absolutely continuous copula, then the following conditions are equivalent:
	\begin{itemize}
		\item[(i)] 	$C(u,v)+ C(u,1-v)-u=0$  for all  $u,v\in [0,1]$.	
		\item[(ii)] $\partial_2   C(u,v) = \partial_2   C(u,1-v)$ for all $u\in (0,1)$ and $v\in (0,1/2)$.	
	\end{itemize}
\end{proposition}
\begin{proof}\quad 
	
	If $(i)$ holds taking partial derivatives with respect to $v$ we get $(ii)$.
	
	Conversely, if $(ii)$ holds, then
	$$C(u,v) +C(u,1-v)=\psi(u)$$ for all $u\in (0,1)$ and $v\in (0,1/2)$, where $\psi(u)$ is a function of $u$. Hence, as $C$ is a copula, taking $v\to0^+$ we have 
	$$C(u,0) +C(u,1)=u=\psi(u)$$ for all $u\in [0,1]$ and so $(i)$ holds for all $u\in [0,1]$ and $v\in [0,1/2)$. Finally, we note that then  $(i)$ also holds for all $u\in [0,1]$ and $v\in [1/2,1]$.
\end{proof}

\begin{proposition}\label{propnew2}
	If the  $sPQD(Z|X)$ (resp. $sNQD(Z|X)$) holds, then $wPQD(Z|X)$ (resp. $wNQD(Z|X)$) holds.
\end{proposition}
\begin{proof}
	Let us define the function
	$$\phi(u,v)= \widehat C(u,v) +\widehat C(u,1-v)$$
	for $u,v\in[0,1]$. 
	Hence $\phi$ is a continuous function satisfying $\phi(u,0)=u$. Moreover, from \eqref{D1bis}, we know that $\phi$ is increasing in $v$ in the interval   $(0,1/2)$ for all $u\in[0,1]$. Therefore, $\phi(u,v)\geq \phi(u,0)=u$ and  \eqref{C-PQDEbis} holds for $v\in [0,1/2)$. Finally, we note that $\phi(u,v)=\phi(u,1-v)$  and so   \eqref{C-PQDEbis} also holds for $u,v\in [0,1]$.
\end{proof}


The relationships between the dependence notions used in this paper are summarized in Table \ref{table1}. 	This table shows that the new positive dependence properties $sPQD(Z|X)$ and $wPQD(Z|X)$ (which only depends on copula properties) imply a positive Spearman rho coefficient (which is an association measure) and a positive Gini measure of association while the classic  $PQDE(Z|X)$ property (which also depends on the marginals) implies a positive Pearson rho coefficient (covariance). To have all these properties together we need the $PQD(X,Z)$ property.

\begin{table}
	\[
	\begin{tabular}{ c c c c c c c}
		$SI_{ST}(Z|X)$ & $\Rightarrow$  &$PQD(X,Z)$ &  $\Rightarrow$ & $PQDE(Z|X)$&$\Rightarrow$ &$Cov(X,Z)\geq 0$\\%
		$\Downarrow$   &&	$\Downarrow$    
		&               &     &               &
		\\%
		$sPQD(Z|X)$&$\Rightarrow$&$wPQD(Z|X)$ &  $\Rightarrow$ & $\rho_C,\gamma_C\geq 0$ & &%
	\end{tabular}
	\]
	\caption{Relationships between positive dependence notions.}\label{table1} 
\end{table}

The following property  gives an interpretation of the $sPQD(Z|X)$ property. It can also be used to construct copulas satisfying property \eqref{D1bis} (see Example \ref{ex3}).

\newpage

\begin{proposition}\label{JM}
	If $C$ is an  absolutely continuous copula and $(U,V)$ has the joint CDF $C$, then the following conditions are equivalent:
	\begin{itemize}
		\item[(i)] 	$C$ satisfies \eqref{D1bis}.	
		\item[(ii)] The PDF $g_u$ of  $(V-1/2|U\leq u)$ is left-skewed, that is, $g_u(-v)\geq g_u(v)$ for all $u\in (0,1)$ and all $v\in (0,1/2)$.	
		\item[(iii)] The PDF $h_u$ of  $(V-1/2|U> u)$ is right-skewed, that is, $h_u(-v)\leq h_u(v)$ for all $u\in (0,1)$ and all $v\in (0,1/2)$.
	\end{itemize} 	
\end{proposition}
\begin{proof}
	The CDF of $(V-1/2|U\leq u)$ is 
	$$G_u(v)=\Pr(V-1/2\leq v|U\leq u)=\frac{C(u,v+1/2)}{u}.$$
	Hence, its PDF is 
	$$g_u(v)=\frac{\partial_2 C(u,v+1/2)}{u}$$
	for all $u\in (0,1)$ and $v\in (-1/2,1/2)$.
	Therefore 
	$$g_u(-v)-g_u(v)=\frac{\partial_2 C(u,-v+1/2)-\partial_2 C(u,v+1/2)}{u}\geq 0$$
	for all $u\in (0,1)$ and $v\in (0,1/2)$ if and only if 
	$$\partial_2 C(u,w)-\partial_2 C(u,1-w)\geq 0$$ 
	for all $u\in (0,1)$ and all $w=-v+1/2\in (0,1/2)$ (i.e. \eqref{D1bis} holds). Hence $(i)$ and $(ii)$ are equivalent.
	
	
	
	To prove that $(i)$ and $(iii)$ are equivalent we note that the survival function of $(V-1/2|U> u)$ is 
	$$\bar H_u(v)=\Pr(V-1/2> v|U> u)=\frac{1-u-v-1/2+C(u,v+1/2)}{1-u}.$$
	Now using \eqref{hatC} it can be written as 
	$$\bar H_u(v)=\frac{\widehat C(1-u,1/2-v)}{1-u}.$$
	Hence, its PDF is 
	$$h_u(v)=\frac{\partial_2 \widehat C(1-u,1/2-v)}{1-u}$$
	for all $u\in (0,1)$ and $v\in (-1/2,1/2)$.
	Therefore 
	$$h_u(v)-h_u(-v)=\frac{\partial_2 \widehat C(u,1/2-v)-\partial_2 \widehat  C(u,-1/2-v)}{1-u}\geq 0$$
	for all $u\in (0,1)$ and $v\in (0,1/2)$ if and only if 
	$$\partial_2 \widehat C(u,w)-\partial_2 \widehat C(u,1-w)\geq 0$$ 
	for all $u\in (0,1)$ and all $w=1/2-v\in (0,1/2)$ (i.e. \eqref{D1} holds). Hence $(i)$ and $(iii)$ are equivalent.	
\end{proof}

\begin{remark}
	We can obtain properties similar to that in Proposition \ref{prop3.14}  for the ICX and ICV orders from \eqref{D1} and skew properties for $G$ but now we need extra assumptions for $\widehat C$ and $g$. For example, for the ICX order we need to assume that $\partial_2  \widehat C(u,v)$ is increasing in $v$ (that is $SI_{ST}(Z|X)$ holds) and that $g(z)\geq g(-z)$ for all $z\geq 0$. We think that these assumptions are too strong and so we do not state  these results. 
\end{remark}

Finally, we note that, to compare $X$ and $X+Z$ when $X$ and $Z$ are dependent, we can also use the expressions for C-convolutions (sums of dependent variables with a copula $C$) given in \cite{CMR11,NS22}. For example, from expression (2.13) in \cite{NS22}, the survival function of $X+Z$ can be obtained from the survival copula as
$$\Pr(X+Z>y)=\int_{-\infty}^{\infty} g(z) \partial_2\widehat C(\bar F(y-z),\bar G(z))dz.$$


Some ordering results can be seen in  \cite{NS22}. In particular, it is easy to see that if $Z\geq 0$, then $X\leq_{ST}X+Z$ as stated in Proposition \ref{prop2.4}.
Results for the ICX order were obtained in \cite{BM98}. For that order, if $F,G$ and $\widehat C$ are known,  we can also use the expression for $E(Z|X>x)$ given in the proof of Proposition \ref{prop3.7} and the result in Proposition \ref{G2}.

\section{Examples}

First we provide an example of a copula that is $wPQD$ but not $PQD$.

\begin{example}\label{ex1}
	Let us assume that $U$ has a standard uniform distribution  and that $(V|U=u)$ follows a beta distribution with PDF
	$$f_u(v)=k_u v^{(0.5-u)^2}(1-v)^{(0.5-u)^2}\ \text{ for } v\in[0,1]$$
	if  $u\in(0,1/2]$ and the following PDF
	$$f^*_u(v)=2-f_u(v)=2-k_u v^{(0.5-u)^2}(1-v)^{(0.5-u)^2}\ \text{ for } v\in[0,1]$$
	if  $u\in(1/2,1)$, where $k_u=1/\beta(1+(0.5-u)^2,1+(0.5-u)^2)$ is the normalizing constant and $\beta(\cdot,\cdot)$ is the beta function. It is easy to see that $f^*_u$ is a proper PDF for all  $u\in(1/2,1)$ since $f_u(v)< 2$ for all $u\in(0,1/2]$ and all $v\in[0,1]$.
	
	Then the PDF of $V$ is 
	\begin{align*}
		f_V(v)
		&=\int_0^{1/2} k_u (v-v^2)^{(0.5-u)^2}du+\int_{1/2}^1 \left(2-k_u (v-v^2)^{(0.5-u)^2}\right)du\\
		&=\int_0^{1/2} k_u (v-v^2)^{(0.5-u)^2}du+1-\int_{0}^{1/2} k_z (v-v^2)^{(0.5-z)^2}dz=1
	\end{align*}
	for all $v\in[0,1]$, where the last integral is obtained by doing the change $z=1-u$. Hence,  $V$ has a standard uniform distribution and the absolutely continuous  CDF $C$ of $(U,V)$ is a proper copula. 
	
	The joint PDF of this copula is 
	\begin{equation*}
		c(u,v)=\left\{\begin{array}{rl}
			f_u(v) , & \text{ for } u\in [0,1/2],\ v\in[0,1];\\
			2-f_u(v) , & \text{ for } u\in (1/2,1],\ v\in[0,1];\\
			0, & \text{ elsewhere. }
		\end{array}%
		\right.
	\end{equation*}
	The plot and the contour (level) plot of $c$ can be seen in Figure \ref{fig1ex1}.

	\begin{figure}[t]
		\begin{center}
			\includegraphics*[scale=0.5]{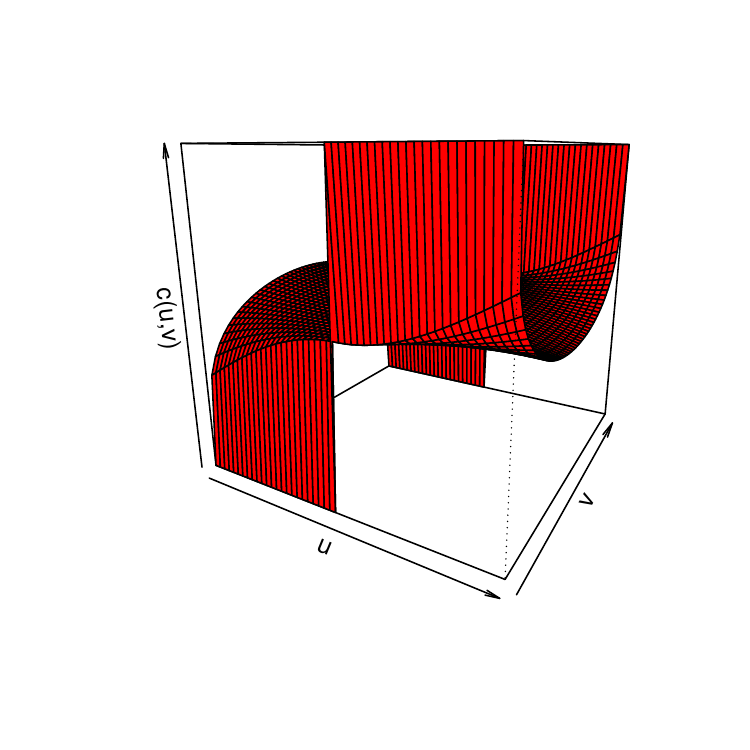}
			\includegraphics*[scale=0.5]{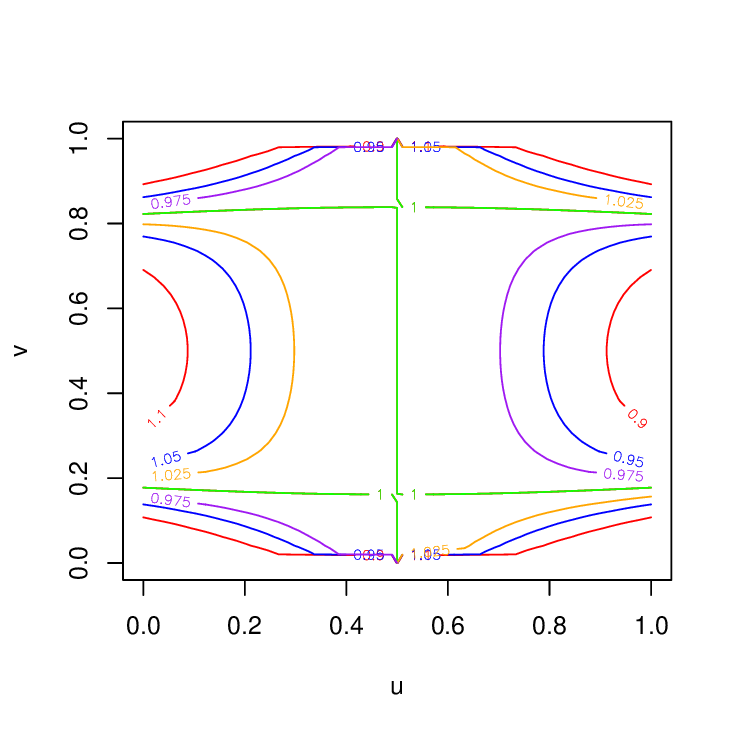}
			\caption{Plot and contour plot of the PDF for the copula  in Example \ref{ex1}.}\label{fig1ex1}
		\end{center}
	\end{figure}

	Hence, the copula $C$ can be obtained for $v\in[0,1]$ as 
	\begin{align*}
		C(u,v)
		&=\int_0^{u}\int_0^{v}  f_x(y)dydx=\int_0^{u} F_x(v)dx
	\end{align*}
	for $u\in[0,1/2]$ and as 
	\begin{align*}
		C(u,v)
		&=\int_0^{1/2}\int_0^{v} f_x(y)dydx+\int_{1/2}^u \int_0^{v} (2-f_x(y))dydx \\
		&=\int_0^{1/2} F_x(v)dx+ 2v(u-0.5)-\int_{1/2}^u F_x(v)dx\\
		&=\int_0^{1/2} F_x(v)dx+ 2v(u-0.5)-\int_{1-u}^{1/2} F_x(v)dx\\
		&= 2uv-v-\int_0^{1-u} F_x(v)dx
	\end{align*}
	for $u\in(1/2,1]$, where $F_u$ is the beta CDF associated to $f_u$ and where we have used that $F_u=F_{1-u}$ for all $u\in(0,1/2)$.

	For this copula we have 
	$$\Pr(V<v|U=u)=\Pr(V>1-v|U=u)$$
	for all $u,v\in (0,1)$ due to the symmetry of the PDF $f_u$ and $f_u^*$ around $1/2$. Hence, the distributions of $(V|U=u)$ and $(1-V|U=u)$ coincide for all $u\in(0,1)$ and so $(U,V)=_{ST}(U,1-V)$. Therefore,
	$$C(u,v)+C(u,1-v)-u=0 $$
	for all  $u,v\in(0,1)$. Hence it is both $wPQD(V|U)$ and $wNQD(V|U)$. Then, from Proposition \ref{PropCov}, we have  $\gamma_C=\rho_C=Cov(U,V)= 0$. However, it is not PQD since  
	$$C(0.5,0.2)=\int_0^{0.5} F_u(0.2) du=0.09442759<0.5\cdot0.2=0.1.$$ $\hfill \square$
\end{example}

Next we provide another example with similar properties but with  a simple (uniform) distribution. This example also proves that $wPQD(Z|X)$ (or $wPQD(Z|X)$) does not imply $Cov(X,Z)\geq 0$.

\begin{example}\label{ex2}
	Let us assume that $(U,V)$ follows a uniform distribution over the support $S=A_1\cup A_2\cup A_3$, where $A_1=(0,1/2)\times (0,1/4)$, $A_2=(0,1/2)\times (3/4,1)$, and $A_3=(1/2,1)\times (1/4,3/4)$, that is,  its joint PDF is  $c(u,v)=2$ for $(u,v)\in S$ (zero elsewhere). The support $S$ can be seen in Figure \ref{fig1ex2}, left.

	\begin{figure}[t]
		\begin{center}
			\includegraphics[scale=0.5]{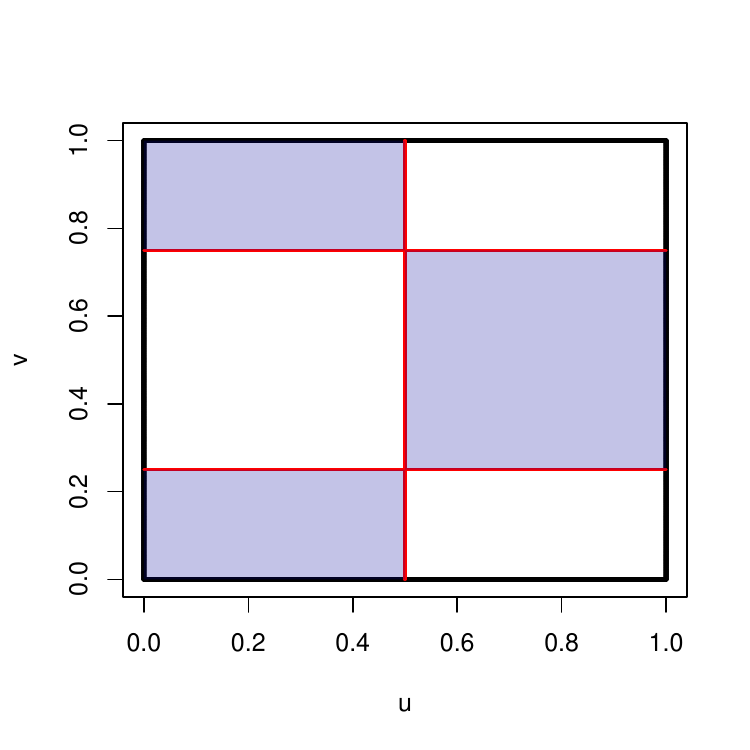}
			\includegraphics[scale=0.5]{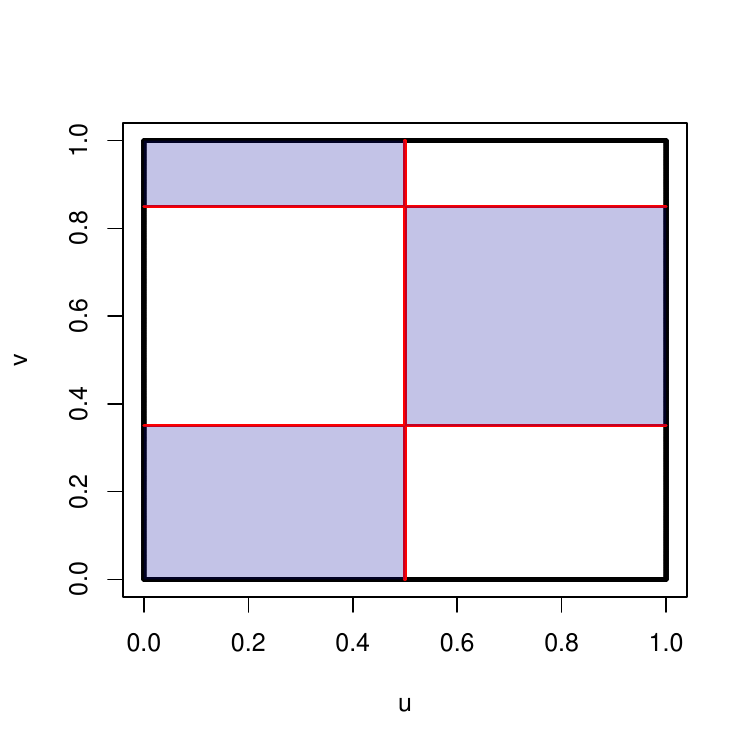}
			\caption{Supports of the (uniform) copula functions in Examples \ref{ex2} (left) and \ref{ex3} (right) for $\delta=0.1$.}\label{fig1ex2}
		\end{center}
	\end{figure}
	
	Then the copula function is 
	\begin{equation*}
		C(u,v)=\left\{\begin{array}{rl}
			2uv , & \text{ for } u\in [0,1/2],\ v\in[0,1/4];\\
			\frac 1 2 u , & \text{ for } u\in [0,1/2],\ v\in(1/4,3/4];\\
			2uv-u , & \text{ for } u\in [0,1/2],\ v\in(3/4,1];\\
			v , & \text{ for } u\in (1/2,1],\ v\in[0,1/4];\\
			\frac 1 2 +2uv -\frac 1 2 u-v, & \text{ for } u\in (1/2,1],\ v\in(1/4,3/4];\\
			u+v-1 , & \text{ for } u\in (1/2,1],\ v\in(3/4,1].\\
		\end{array}%
		\right.
	\end{equation*}
	It is easy to see that it is a proper copula function. 
	
	
	Clearly, this copula is neither PQD nor NQD. For example,
	$$C(1/2,1/4)=\frac 14 >\frac 18=\frac 12\cdot \frac 14$$
	and
	$$C(1/2,3/4)=\frac 14 <\frac 38=\frac 12\cdot \frac 34.$$
	However, it is easy to see that $(U,V)=_{ST}(U,1-V)$ and so 
	$$C(u,v)+C(u,1-v)-u=0$$
	holds for all $u,v\in[0,1]$. Therefore, it is both  $wPQD(V|U)$ and  $wNQD(V|U)$. 
	Then, from Proposition \ref{propnew1}, it satisfies 
	$$\partial_2  C(u,v) = \partial_2  C(u,1-v)  \text{ for all } u\in (0,1) \text{ and }v\in (0,1/2)$$
	and so 	it is both  $sPQD(V|U)$ and  $sNQD(V|U)$.  It is also easy to see that $\gamma_C=\rho_C=Cov(U,V)=0$ (as stated in Proposition \ref{PropCov}).

	We can also confirm that  it is neither  	$wPQD(U|V)$ nor $wNQD(U|V)$. For example,
	$$ C(u,v)+ C(1-u,v)-v=2uv>0$$
	for $u\in (0,1/2)$ and  $v\in(0,1/4)$
	and
	$$C(u,v)+C(1-u,v)-v=-2u(1-v)<0$$
	for $u\in (0,1/2)$ and  $v\in(3/4,1)$.
	Hence, this example  proves that this concept is not symmetric with respect to the two different conditioning.
	
	Let us consider now $(X,Z)$ with $X=U$ and $Z=V^2$. Clearly, $C$ is the copula of $(X,Z)$ and so it satisfies both  $wPQD(Z|X)$ and  $wNQD(Z|X)$.  Then $E(X)=1/2$ and $E(Z)=E(V^2)=1/3$ and a straightforward calculation shows that $E(XZ)=E(UV^2)=29/192$. Hence,
	$$Cov(X,Z)= \frac{29}{192}-\frac 1 2 \cdot \frac 1 3 =-\frac 1 {64}<0.$$
	Therefore $wPQD(Z|X)$ does not imply $Cov(X,Z)\geq 0$. Even more, the property $wPQD(Z|X)$ does not imply $PQDE(Z|X)$ (since this last property implies $Cov(X,Z)\geq 0$). The same holds for the stronger condition $sPQD(Z|X)$. $\hfill \square$
\end{example}

In the last example we use Proposition \ref{JM} to build a copula satisfying \eqref{D1bis}. It is a modification of the uniform copula used in the preceding example.

\begin{example}\label{ex3}
	Let us assume that $(U,V)$ follows a uniform distribution over the support $S_\delta=A_1\cup A_2\cup A_3$, where $A_1=(0,1/2)\times (0,1/4+\delta)$, $A_2=(0,1/2)\times (3/4+\delta,1)$, and $A_3=(1/2,1)\times (1/4+\delta,3/4+\delta)$ for $\delta\in[0,1/4]$, that is with joint PDF $c(u,v)=2$ for $(u,v)\in S$ (zero elsewhere). The support is plotted  in Figure \ref{fig1ex2} for $\delta=0$ (left, Example \ref{ex2}) and $\delta=0.1$ (right). It is easy to see that the PDF  $c$ leads to a proper copula $C$ for all $\delta\in[0,1/4]$. If we consider now the conditional random variable $(V|U\leq u)$ for $u\in(0,1/2)$, it  has the PDF
	$$g(v|U\leq u)=2 \text{ for } v\in(0,1/4+\delta)\cup (3/4+\delta,1),$$
	which is left-skewed with respect to $1/2$ since $$g(v+1/2|U\leq u)\leq g(-v+1/2|U\leq u)$$ for $v\in (0,1/2)$.
	Analogously, for $u\in(1/2,1)$, it  has the PDF
	$$g(v|U\leq u)=\left\{\begin{array}{rl}
		1/u, & \text{ for }  v\in[0,1/4+\delta];\\
		2 -1/u, & \text{ for } v\in(1/4+\delta,3/4+\delta);\\
		1/u, & \text{ for }  v\in[3/4+\delta,1];\\
	\end{array}%
	\right.$$
	which is also  left-skewed with respect to $1/2$ since 
	$2 -1/u\leq 1/u$ for $u\in(1/2,1)$. Therefore, from Proposition \ref{JM},   $sPQD(V|U)$ holds for all $\delta\in[0,1/4]$. As a consequence,  $wPQD(V|U)$, $\rho_C\geq 0$ and $\gamma_C\geq 0$  also hold. 
	
	Moreover, it is easy to see that $E(V|U\leq u)\leq E(V)=1/2$ for all $u\in(0,1)$ and all $\delta\in[0,1/4]$. Hence, $PQDE(V|U)$ holds and so $Cov(U,V)\geq 0$. 
	$\hfill \square$ 
\end{example}

\section{Conclusions}

We have obtained several conditions to stochastically compare $X$ and $X+Z$ in the convex, increasing convex and increasing concave orders, where $X$ and $Z$ can be dependent. The conditions  are based on the recent results obtained in \cite{GHW24}. The main advantage of our conditions is that they are applied to the copula function (dependence structure) and to the marginal distribution of $Z$. Then, in many cases, they can be checked in a simple way and they allow to get distribution-free comparisons based on copula dependence properties and skew properties of the marginal distribution of $Z$.

The conditions on the copula function lead to positive (or negative) dependence conditions which are related with other well known dependence properties (as PQD or PQDE). So,  from the basic results of copula theory, we know that they hold for several relevant copulas. Even more, they may hold even if the covariance is negative (see Examples \ref{ex1} and \ref{ex2}), that is, they are weak dependence  conditions.  Moreover, we provide  simple interpretations for both dependence properties that can be used to build copulas fulfilling these properties. 

There are several tasks for future research projects. The main ones could be to extend these results to stronger stochastic dominance concepts as the hazard rate or the likelihood ratio orders.  
We should also study if these dependence notions hold for copulas that do not have  the PQD property. Finally, these properties should be applied to real data with different dependence structures in several fields (actuarial sciences, reliability theory, risk theory, etc.).  








\section*{Acknowledgements}
JN and JMZ acknowledge the partial support of the Ministerio de Ciencia e Innovación of Spain in the project PID2022-137396NB-I00, funded by MICIU/AEI/10.13039/501100011033 and by 'ERDF A way of making Europe'.



%
%
%

\end{document}